\documentclass[11pt]{amsart}

\usepackage{amscd,amsmath,amssymb,fancyhdr,color}
\usepackage[utf8]{inputenc}
\usepackage{amsfonts}
\usepackage{amsthm}
\usepackage{accents}
\usepackage{graphicx}
\usepackage{float}
\usepackage{subcaption}
\usepackage{verbatim}
\usepackage{indentfirst}
\usepackage{dsfont}
\usepackage{tikz}
\usepackage{tikz-cd}
\usetikzlibrary{matrix}
\usepackage[all]{xy}
\usepackage{enumerate}
\usepackage{extarrows}
\usepackage{csquotes}

\usepackage{scalerel,stackengine}
\stackMath
\newcommand\reallywidehat[1]{%
	\savestack{\tmpbox}{\stretchto{%
			\scaleto{%
				\scalerel*[\widthof{\ensuremath{#1}}]{\kern-.6pt\bigwedge\kern-.6pt}%
				{\rule[-\textheight/2]{1ex}{\textheight}}
			}{\textheight}%
		}{0.5ex}}%
	\stackon[1pt]{#1}{\tmpbox}%
}

\usepackage{faktor}

\usepackage{BOONDOX-uprscr}

\usepackage[backref=page]{hyperref}
\renewcommand*{\backref}[1]{}
\renewcommand*{\backrefalt}[4]{%
	\ifcase #1 (Not cited.)%
	\or        (Cited on page~#2.)%
	\else      (Cited on pages~#2.)%
	\fi}

\hypersetup{
	colorlinks   = true,
	citecolor    = magenta
}

\newcommand{\K}{K\"ahler}

\DeclareMathOperator{\reg}{reg}
\DeclareMathOperator{\sing}{sing}

\numberwithin{equation}{section}

\def\eqref#1{(\ref{#1})}

\newcommand{\R}{{\mathbb R}}

\newcommand{\ei}{\textup{i}}

\def\1{\sqrt{-1}\:}

\newcommand{\cntrct}                
{\hspace{2pt}\raisebox{1pt}{\text{$\lrcorner$}}\hspace{2pt}}

\renewcommand{\dim}{\operatorname{dim}}

\newcommand{\Deck}{\operatorname{Deck}}

\renewcommand{\Re}{\operatorname{Re}}

\newcommand{\ie}{{\em i.e. }}

\newcounter{Mycounter}[section]
\newcounter{lemma}[section]
\newcounter{claim}[section]
\newcounter{sublemma}[section]
\newcounter{corollary}[section]
\newcounter{theorem}[section]
\newcounter{conjecture}[section]
\newcounter{proposition}[section]
\newcounter{definition}[section]
\newcounter{example}[section]
\newcounter{remark}[section]
\newcounter{problem}[section]
\newcounter{question}[section]
\makeatletter

\@addtoreset{equation}{section}

\@addtoreset{footnote}{section}

\makeatother

\usetikzlibrary{arrows,chains,matrix,positioning,scopes}

\makeatletter
\tikzset{join/.code=\tikzset{after node path={%
			\ifx\tikzchainprevious\pgfutil@empty\else(\tikzchainprevious)%
			edge[every join]#1(\tikzchaincurrent)\fi}}}
\makeatother

\tikzset{>=stealth',every on chain/.append style={join},
	every join/.style={->}}

\makeatletter
\newtheorem*{rep@theorem}{\rep@title}
\newcommand{\newreptheorem}[2]{%
	\newenvironment{rep#1}[1]{%
		\def\rep@title{\ref{##1}}%
		\begin{rep@theorem}}%
		{\end{rep@theorem}}}
\makeatother

\newreptheorem{theorem}{Theorem}

\newtheoremstyle{named}{}{}{\itshape}{}{\bfseries}{.}{.5em}{\thmnote{#3's }#1}
\theoremstyle{named}

\begin{document}
	
	\newpage
	
	\title[lcK spaces and proper open morphisms]{Locally conformally K\" ahler spaces and proper open morphisms}

	\author{Ovidiu Preda}
	\address{Ovidiu Preda \newline
		\textsc{\indent University of Bucharest, Faculty of Mathematics and Computer Science\newline 
			\indent 14 Academiei Str., Bucharest, Romania\newline
			\indent \indent and\newline
			\indent Institute of Mathematics ``Simion Stoilow'' of the Romanian Academy\newline 
			\indent 21 Calea Grivitei Street, 010702, Bucharest, Romania}}
	\email{ovidiu.preda@fmi.unibuc.ro; ovidiu.preda@imar.ro}
	
	\author{Miron Stanciu}
	\address{Miron Stanciu \newline
		\textsc{\indent University of Bucharest, Faculty of Mathematics and Computer Science\newline 
			\indent 14 Academiei Str., Bucharest, Romania\newline
			\indent \indent and \newline
			\indent Institute of Mathematics ``Simion Stoilow'' of the Romanian Academy\newline 
			\indent 21 Calea Grivitei Street, 010702, Bucharest, Romania}}
	\email{miron.stanciu@fmi.unibuc.ro; miron.stanciu@imar.ro}
	
	\thanks{Ovidiu Preda was partially supported by a grant of Ministry of Research and Innovation, CNCS - UEFISCDI, project no.
		PN-III-P1-1.1-TE-2021-0228, within PNCDI III. \\
		\indent Miron Stanciu was partially supported by a grant of Ministry of Research and Innovation, CNCS - UEFISCDI, project no.
		PN-III-P4-ID-PCE-2020-0025, within PNCDI III. \\\\[.1cm]
	}
	
	\date{\today}

	\begin{abstract}
		In this paper, we prove a stability result for the non-\K\ geometry of locally conformally \K\ (lcK) spaces with singularities. Specifically, we find sufficient conditions under which the image of an lcK space by a holomorphic mapping also admits lcK metrics, thus extending a result by Varouchas about \K\ spaces. 
	\end{abstract}
	
	\maketitle
	
	\hypersetup{linkcolor=blue}
	\tableofcontents

	\section{Introduction}
	
	While the strongest geometric results on complex manifolds may be obtained in the pure \K \ setting, the requirement of the existence of such a metrics in the compact case imposes great topological and geometric restrictions, and thus \K \ manifolds are relatively rare. That is why in the last decades there have been many efforts to find suitable replacements by relaxing, in various ways, the \K \ condition, and looking at non-\K \ Hermitian metrics whose existence is more common but can also lead to nice properties for the manifold and, ideally, classification results. One of the most intensely studied metrics are \textit{locally conformally \K} (\textit{lcK} for short).
	
	A Hermitian metric $\omega$ on the complex manifold $X$ is called \textit{locally conformally K\" ahler (lcK)} if for every point $x\in X$, there exists an open neighborhood $U\ni x$ and a smooth function $f:U\rightarrow \mathbb{R}$ such that $d(e^{-f}\omega)=0$ \ie $e^{-f} \omega$ is \K \ on $U$. This is equivalent to $d\omega=\theta\wedge\omega$, where $\theta$ is a closed 1-form, called the Lee form of the lcK metric $\omega$. If the Lee form $\theta$ is exact, which is equivalent to saying that the function $f$ can be defined globally on $X$, then $\omega$ is called \textit{globally conformally K\" ahler (gcK)}. If $\omega$ is lcK, but not gcK, we call it \textit{pure lcK}.
	These metrics were first defined and studied by Vaisman \cite{vaisman1976}, where he also proved a characterization theorem for lcK manifolds: a complex manifold $X$ admits an lcK metric if and only if its universal cover $\widetilde{X}$ admits a K\" ahler metric $\widetilde{\omega}$ on which the deck group acts by homotheties \ie such that for every $\gamma\in \Deck_{X}(\widetilde{X})$, we have $\gamma^*\widetilde{\omega}=e^{c_\gamma}\widetilde{\omega}$. Some years later, in \cite{vaisman1980} he proved that K\" ahler and pure lcK metrics cannot coexist on compact complex manifolds, with respect to the same complex structure. This is one motivation for studying lcK structures preferentially on compact manifolds, as here we have a clear separation between \K \ and lcK geometry. Since Vaisman published these two papers, there has been great progress in understanding compact lcK manifolds and many useful theorems were obtained. For a recent coprehensive overview of the development of lcK geometry, one may check \cite{OV_book}.  
	
	In contrast to the abundance of results about K\" ahler manifolds, the lack of an all-around good definition of $(p, q)$ differential forms on singular spaces (see \cite{michalek} for a comparison study of different generalizations) made the study of K\" ahler spaces more difficult. Grauert \cite{grauert} was the first to give a definition for K\" ahler metrics on complex analytic spaces. A more restrictive definition was used by Varouchas in \cite{var84}, \cite{var86}, \cite{var89} to obtain some results about modifications of \K \ spaces. Among them, he proves that if some conditions are satisfied, then the image of K\" ahler space by a holomorphic, proper mapping also admits K\" ahler metrics. 
	
	In \cite{PS21}, we adapted Grauert's idea of using families of plurisubharmonic functions and compatibility conditions to define lcK metrics on complex analytic spaces, and to prove that the characterization theorem involving the universal cover is still true for lcK spaces, exactly in the same form. Also, in \cite{PS22}, we proved that Vaisman's theorem on the dichotomy \textit{K\" ahler -- lcK} remains true for compact lcK spaces, with the additional assumption that the space is locally irreducible, and also gave an example which shows that the local irreducibility condition cannot be dropped. Although all the results of \cite{PS21} and \cite{PS22} are proved using Grauert's definition of K\" ahler metrics (and its adaptation to lcK metrics), they all remain true for Varouchas' definition, with the same proof, the additional condition of pluriharmonic differences being easily verified.
	
	In this paper, we give an lcK version of Varouchas' stability results from \cite{var89}. More precisely, we prove:
	
	\begin{reptheorem}{main_theorem}
		Let $(X,\omega,\theta)$ be an lcK space of pure dimension and $X^\prime$ be a normal space, such that there exists $p:X\rightarrow X^\prime$ holomorphic, proper, open and surjective. Assume that $\ker p_* \subset \ker \theta$. 
		Then, $X^\prime$ also admits lcK metrics.
	\end{reptheorem}
	
	The strategy for our proof is to lift $p$ to a morphism $\widetilde{p}:\widetilde{X}\rightarrow \widetilde{X}^\prime$ from a covering space $\widetilde{X}$ of $X$ onto the universal cover $\widetilde{X}^\prime$ of $X^\prime$, and then make use of Varouchas' methods \cite{var89} for $\widetilde{p}$ to obtain a K\" ahler metric on $\widetilde{X}^\prime$. As we need to integrate differential forms on the fibers of $\widetilde{p}$, these must be compact and K\" ahler for the method to work. Thus, for $\widetilde{p}$ to still be a proper mapping and its fibers to be K\" ahler, we need to impose the additional condition $\ker p_* \subset \ker \theta$. This is done so that $\Deck_{X^\prime}(\widetilde{X}^\prime)$ acts by positive homotheties on the newly constructed K\" ahler metric on $\widetilde{X}^\prime$. Finally, the characterization theorem for lcK spaces mentioned above yields that $X^\prime$ has lcK metrics.
	
	\vspace{10pt}
	
	In Section \ref{sec:prelim} we give all the definitions and the results we use. Section \ref{sec:main} is devoted to proving our new result.
	
	\section{Preliminaries}\label{sec:prelim}
	\subsection*{K\" ahler and lcK metrics}
	
	Firstly, we recall the definitions for \K\ and lcK metrics on complex analytic spaces.
	
	\begin{definition}\label{K+lcK}
		Let $X$ be a complex analytic space. 
		\begin{enumerate}
			\item[\textbf{(K)}] A \textit{K\" ahler metric} on $X$ is the equivalence class $\reallywidehat{(U_i,\varphi_i)_{i\in I}}$ of a family such that $(U_i)_{i\in I}$ is an open cover of $X$, $\varphi_i:U_i\rightarrow \mathbb{R}$ is $\mathcal{C}^\infty$ and strictly psh, and $\varphi_i-\varphi_j=\Re(h_{ij})$ on $U_{ij}=U_i\cap U_j$, for every $i,j\in I$, where $h_{ij}$ is holomorphic.
			Two such families are equivalent if their union verifies the compatibility condition on the intersections, described above.
			
			\item[\textbf{(lcK)}] An \textit{lcK metric} on $X$ is the equivalence class $\reallywidehat{(U_i,\varphi_i,f_i)_{i\in I}}$ of a family such that $(U_i)_{i\in I}$ is an open cover of $X$, $\varphi_i:U_i\rightarrow \mathbb{R}$ is $\mathcal{C}^\infty$ and strictly psh, $f_i:U_i\rightarrow\mathbb{R}$ is smooth, and $e^{f_i-f_j}\varphi_i-\varphi_j=\Re(h_{ij})$ on $U_{ij}=U_i\cap U_j$, for every $i,j\in I$. 
			As before, two such families are equivalent if their union verifies the compatibility condition on the intersections. 
		\end{enumerate}
	\end{definition}
	
	\begin{remark}
		The definition of K\" ahler metrics on complex spaces was first introduced by Grauert in \cite[p.346]{grauert}. In his definition, it is required only that $\varphi_i-\varphi_j=\Re(h_{ij})$, where $h_{ij}$ is holomorphic on $U_{ij}\cap X_{\reg}$. 
		
		\ref{K+lcK} is Varouchas' definition \cite[p.23]{var89}. It requires that $\varphi_i-\varphi_j=\Re(h_{ij})$, where $h_{ij}$ is holomorphic on $U_{ij}$ (including the singular points). Hence, it is more restrictive that the one given by Grauert, but they coincide if $X$ is normal. Since we use extensively Varouchas' results from \cite{var89}, we also follow his definition of K\" ahler metric throughout this article.
	\end{remark}
	
	\vspace{10pt}
	
	For lcK forms on singular spaces we also want to define the analogue of its associated Lee form. For this, we have the following:
	
	\begin{definition}\label{TC-1-form-definition}
		\begin{itemize}
			\item Let $X$ be a topological space and consider $(U_i,f_i)_{i\in I}$, consisting of an open cover $(U_i)_{i\in I}$ of $X$ and a family of continuous functions $f_i:U_i\rightarrow\mathbb{R}$ such that $f_i-f_j$ is locally constant on $U_i\cap U_j$, for all $i,j\in I$. The class  
			\[
			\theta =\reallywidehat{(U_i,f_i)_{i\in I}}\in  \check{\mathrm{H}}^0\left(X,\faktor{\mathscr{C}}{\underline{\R}}\right)
			\]
			is called a \textit{topologically closed $1$-form} (\textit{TC $1$-form}). 
			\item We say that a TC 1-form $\theta$ is \textit{exact} if $\theta = \widehat{(X, f)}$ for a continuous function $f:X \rightarrow \mathbb{R}$. In this case, we make the notation $\theta = df$.
			\item 	Let $\omega=\reallywidehat{(U_i,\varphi_i,f_i)_{i\in I}}$ be an lcK metric on a complex space $X$. Then, the TC 1-form $\theta=\reallywidehat{(U_i,f_i)_{i\in I}}$ is called the \textit{Lee form} of $\omega$. If $\theta$ is exact, then $\omega$ is called \textit{globally conformally K\" ahler (gcK)}.
		\end{itemize}
	\end{definition}
	
	\vspace{10pt}
	
	\subsection*{Pushforward and stability} 
	
	We assemble below a few results we will need later. 
	
	The first is a theorem (\cite[p.330 (III)]{remmert1957}) which gives necessary and sufficient conditions under which a holomorphic mapping of complex spaces has pure and equal dimensional fibers.
	
	\begin{theorem}
		\label{remmert_open}
		Let $p:X\rightarrow Y$ be a holomorphic and surjective mapping of complex spaces, with $Y$ locally irreducible. Then, $p$ is an open mapping if and only if $\dim_x p^{-1}(p(x))=\dim_x X - \dim_{p(x)}Y$ for every $x\in X$.
	\end{theorem}
	
	\bigskip
	
	The next result (\cite[Chap. II, Lemma 3.1.2]{var89} combined with \cite[Chap. I, Section 3.3]{var89}) shows that analytical properties of functions are preserved by pushforward through an open finite map.
	
	\begin{lemma}\label{integration_lemma}
		Consider $p:X\rightarrow X^\prime$ a finite, open and surjective morphism of complex spaces. 
		If $\varphi$ is psh, strictly psh, holomorphic or pluriharmonic on $X$, then 
		$$p_*\varphi(x^\prime)=\sum_{x\in p^{-1}(x^\prime)}\varphi(x)$$ 
		has the corresponding properties on $X^\prime$. 
	\end{lemma}
	
	\bigskip
	
	As to the pushforward through a map which is not finite, the summation above is naturally replaced by integration on the fibers. Firstly, we need the following sufficient condition for geometric flatness \cite[Section 3.3, Prop. 3.3.1]{var89}.
	
	\begin{proposition}\label{prop.3.3.1}
		Suppose that $p:X\rightarrow X^\prime$ is a morphism of complex spaces such that, for some fixed $m\geq 0$, the following conditions are verified:
		\begin{itemize}
			\item $\pi$ is proper, open and surjective;
			\item all fibers of $\pi$ are of pure dimension $m$;
			\item $X^\prime$ is reduced;
			\item $X^\prime$ is normal.
		\end{itemize}
		
		Then $\pi$ is geometrically flat.
	\end{proposition}
	
	\vspace{10pt}
	
	Geometric flatness is a notion we do not use directly, but we need it for connecting the previous proposition with the next one, which is the part that we need from \cite[Chap. I, Proposition 3.4.1]{var89}, combined with \cite[Th\' eor\` eme principal]{barlet&varouchas89}. It says that positivity and holomorphicity are again preserved by pushforward via a holomorphic map with good properties. Also, in what follows, we use the definitions of differential forms on complex spaces as given in \cite[Chap. I, Section 1]{var89}. 
	
	\begin{proposition}\label{prop.3.4.1}
		Consider $p:X\rightarrow X^\prime$ holomorphic, proper, geometrically flat, with $m$-dimensional fibers and $\varphi\in A^{m,m}(X)$. Define 
		$$p_*\varphi(x^\prime)=\int_{p^{-1}(x^\prime)}\varphi.$$
		Then,
		\begin{enumerate}[i)]
			\item if $\varphi=\overline{\varphi}$ and $\ei\partial\overline{\partial}\varphi\geq 0$, then $p_*\varphi$ is psh.
			\item  if $\varphi=\overline{\varphi}$ and $\ei\partial\overline{\partial}\varphi \gg 0$, then $p_*\varphi$ is s.psh.
			\item  if $\overline{\partial}\varphi=0$, then $p_*\varphi$ is holomorphic.
		\end{enumerate}
	\end{proposition}
	
	\vspace{10pt}
	
	The key result in proving the stability theorems on projections of \K\ spaces is the following (\cite[Theorem 3]{var89}):
	
	\begin{theorem}\label{thm2Varouchas}
		Let $(X,\omega)$ be a \K\ space and $m\geq 0$ an integer. Then there exist open sets $U_\alpha\subset X$ ($\alpha \in A$) and $U_{\alpha\beta}^j\subset U_\alpha \cap U_\beta$ ($j\in J_{\alpha\beta}$), which depend only on $X$ and $m$ alone such that:
		\begin{enumerate}[i)]
			\item Any compact $m$-dimensional analytic subset of $X$ is contained in some $U_\alpha$.
			\item Any compact $m$-dimensional analytic subset of $U_\alpha\cap U_\beta$ is contained in some $U_{\alpha\beta}^j$.
			\item There exist elements $\chi_\alpha\in A^{m,m}(U_\alpha,\mathbb{R})$ such that $$\omega^{m+1}=\ei \partial \overline{\partial} \chi_\alpha.$$
			\item There exist elements $\tau_{\alpha\beta}^j\in A^{m,m}(U_{\alpha\beta}^j)$ such that $$\overline{\partial} \tau_{\alpha\beta}^j=0 \text{\ and \ } (\chi_\alpha-\chi_\beta)_{\restriction U_{\alpha\beta}}=\tau_{\alpha\beta}^j+\overline{\tau}_{\alpha\beta}^j.$$
			\item The $\tau_{\alpha\beta}^j$ are $\overline{\partial}$-closed representatives of elements $\xi_{\alpha\beta}^j\in H^m(U_{\alpha\beta}^j,\Omega^m)$.
		\end{enumerate}
	\end{theorem}

	\section{The main result}\label{sec:main}
	
	In this section, we prove our main result on the existence of lcK metrics on images of lcK spaces. In the particular case of finite mappings, our previous result \cite[Thm.4.1]{PS21} says that if $p:X\rightarrow X^\prime$ is holomorphic and finite, and $X^\prime$ admits lcK metrics, then $X$ is also admits lcK metrics. Our theorem, in the case of 0-dimensional fibers, is a kind of reciprocal of this result, in the sense of giving information about the image of an lcK space, instead of the preimage. However, we need some additional conditions to be verified for our proof to work, for which we introduce the following notations: for a mapping $p:X\rightarrow X^\prime$, denote $p_*:\pi_1(X)\rightarrow\pi_1(X^\prime)$ the induced morphism, and for a TC 1-form $\theta$ on a complex space $X$, we denote $$\ker\theta=\left\{ \gamma\in\pi_1(X) \mid \int_\gamma \theta=0\right\},$$ where the integral $\int_\gamma \theta$ is defined as in \cite{PS21}.
	
	\begin{theorem}\label{main_theorem}
		Let $(X,\omega,\theta)$ be an lcK space of pure dimension and $X^\prime$ be a normal space, such that there exists $p:X\rightarrow X^\prime$ holomorphic, proper, open and surjective. Assume that $\ker p_* \subset \ker \theta$. 
		Then, $X^\prime$ also admits lcK metrics.
	\end{theorem}
	\begin{proof}
		Denote $\pi^\prime:\widetilde{X^\prime}\rightarrow X^\prime$ the universal cover of $X^\prime$ and consider $\widetilde{X}=X\times_{X^\prime}\widetilde{X^\prime}$ to be the pull-back of the universal cover $\pi^\prime:\widetilde{X^\prime}\rightarrow X^\prime$ along $p$. Then, we have the following commutative diagram:
		\begin{figure}[H]
			\centering
			\label{diag:thmEmbeddingOutline}
			\begin{tikzpicture}
				\matrix (m) [matrix of math nodes,row sep=3em,column sep=4em,minimum width=2em]
				{
					\widetilde{X} & \widetilde{X^\prime} \\
					(X, \omega, \theta) &  X^\prime \\};
				\path[-stealth]
				(m-1-1) edge node [above] {$\widetilde{p}$} (m-1-2)
				(m-1-1) edge node [left] {$\pi$} (m-2-1)
				(m-1-2) edge node [right] {$\pi^\prime$} (m-2-2)
				(m-2-1) edge node [below] {$p$} (m-2-2);
			\end{tikzpicture}
		\end{figure}
		\noindent where $\widetilde{p}$ is also holomorphic, proper, open and surjective, and $\pi$ is a cover of $X$. Moreover, since $X^\prime$ is normal, $\widetilde{X^\prime}$ is also normal. 	Since $X^\prime$ is normal, it is locally irreducible, hence, by Remmert's open mapping  \ref{remmert_open}, all the fibers of $p$ have pure dimension $m$. Also, by taking the pull-back metric,  $(\widetilde{X},\pi^*\omega,\pi^*\theta)$ is an lcK space. 
		
		Next, we should note that our assumption $\ker p_* \subset \ker \theta$ is equivalent to $\pi^*\theta$ being exact by elementary covering space theory. Indeed, for any $\widetilde{\gamma} \in \pi_1(\widetilde{X})$, we have firstly that $\widetilde{p}_* \widetilde{\gamma} = 0$ as $\widetilde{X^\prime}$ is simply connected, so $\pi'_*\widetilde{p}_* \widetilde{\gamma} = 0$. Equivalently, as the diagram is commutative, $\pi_* \widetilde{\gamma} \in \ker p_*$. On the other hand $\int_{\widetilde{\gamma}}\pi^*\theta = \int_{\pi_*\widetilde{\gamma}} \theta$, so $\pi^* \theta$ is exact \ie $\int_{\widetilde{\gamma}}\pi^*\theta = 0$ for any $\widetilde{\gamma} \in \pi_1(\widetilde{X})$ if and only if $\ker p_* \subset \ker \theta$.
		
		Thus $\pi^*\theta=d\widetilde{f}$ for a smooth function $\widetilde{f}$ on $\widetilde{X}$. This $\widetilde{f}$ also verifies $\widetilde{f}\circ\xi=\widetilde{f}-c_\xi$ for each $\xi\in H:=\Deck_{X}(\widetilde{X})$, where $c_\xi\in\mathbb{R}$.
		Subsequently, $e^{-\widetilde{f}}\pi^*\omega$ is a \K\ metric on $\widetilde{X}$. 
		If $\omega=\reallywidehat{(U_i,\varphi_i,f_i)_{i\in I}}$, then  
		$$\widetilde{\omega}=e^{-\widetilde{f}}\pi^*\omega=\reallywidehat{(U_i^\eta,\varphi_i^\eta)_{i \in I, \eta\in H}},$$
		where $\pi^{-1}(U_j)=\cup_{\eta\in H}U_i^\eta$ is a disjoint union, and $\varphi_i^\eta=e^{f_i\circ\pi-\widetilde{f}}\varphi_i\circ \pi_{\restriction U_i^\eta}$. A simple calculation shows that $\xi^*\widetilde{\omega}=e^{c_\xi}\widetilde{\omega}$ for each $\xi\in H$.
		
		\vspace{10pt}
		
		Now, the proof splits into two cases for which we need different tools, so we treat them separately. 
		
		\vspace{5pt}
		
		\textbf{Case 1: $m=0$}. This means that $\pi$ is a finite mapping.		
		For this step of the proof, we use the methods of Varouchas \cite{var89} to construct a \K\ metric on $\widetilde{X^\prime}$. 
		Note that since $\widetilde{p}$ is holomorphic, finite, open and surjective, it is a ramified covering, so there exists an analytic subset with empty interior $\widetilde{R}\subset \widetilde{X}$ such that $\widetilde{p}_{\restriction \widetilde{X}\setminus \widetilde{R}}:\widetilde{X}\setminus \widetilde{R}\rightarrow \widetilde{X^\prime}\setminus \widetilde{R^\prime}$ is an unramified covering of finite degree $k$.
		For every $(i,\eta)\in I\times H$, consider $\widetilde{p}_*\varphi_i^\eta:V_i^\eta=\widetilde{p}(U_i^\eta)\rightarrow\mathbb{R}$ to be the unique continuous function for which 
		$$\widetilde{p}_*\varphi_i^\eta(x^\prime)=\sum_{x\in\widetilde{p}^{-1}(x^\prime)}\varphi_i^\eta(x)$$
		on $\widetilde{X}\setminus \widetilde{R}$. By \cite[Lemma 3.1.2]{var89}, the functions $\{\widetilde{p}_*\varphi_i^\eta\}_{i\in I, \eta\in H}$ are strictly psh. They are also continuous, of class $\mathcal{C}^\infty$ outside $R^\prime$, and the differences are pluriharmonic outside $R^\prime\cup \widetilde{p}(X_{\sing})$. Moreover, we have $\xi^*\widetilde{p}_*\varphi_i^\eta=e^{c_\xi}\widetilde{p}_*\varphi_i^{\xi^{-1}\eta}$ for every $\xi\in\Deck_{X^\prime}(\widetilde{X^\prime})\simeq\Deck_{X}(\widetilde{X})=H$.
		Next, we apply \cite[Thm.1]{var89} to obtain a \K\ metric $$\widetilde{\tau}^\prime=\reallywidehat{(V_i^\eta,\psi_i^\eta)_{i\in I, \eta\in H}}$$ 
		on $\widetilde{X^\prime}$, with $\mathcal{C}^\infty$ strictly psh functions $\psi_i^\eta$. Since for a fixed $i \in I$, the family of open sets $\{\widetilde{p}(U_i^\eta)\}_{\eta\in H}$ are mutually disjoint, this can be done such that the property $\psi_i^\eta\circ\xi=e^{c_\xi}\psi_i^{\xi^{-1}\eta}$ for every $\eta\in H$ is verified by these new psh functions. 
		Finally, this shows that for every $\xi\in H=\Deck_{X^\prime}(\widetilde{X^\prime})$, we have $\xi^*\widetilde{\tau}^\prime=e^{c_\xi}\widetilde{\tau}^\prime$, and by \cite[Thm.3.10]{PS21}, $X^\prime$ admits lcK metrics.  
		
		\vspace{5pt}
		
		\textbf{Case 2: $m\geq 1$}. There exists an open cover $(V_j)_{j\in J}$ of $X^\prime$ such that:
		\begin{itemize}
			\item $(\pi^\prime)^{-1}(V_j)=\bigcup_{\eta\in H} V_j^\eta$ for every $j\in J$ and for every $j\in J$, $(V_j^\eta)_{\eta\in H}$ are mutually disjoint, and $\xi^{-1} (V_j^\eta)=V^{\xi^{-1}\eta}$ for any $\xi,\eta\in H$
			\item $U_j^\eta:=\widetilde{p}^{-1}(V_j^\eta)$, for every $j\in J, \eta\in H$, and $\xi^{-1} (U_j^\eta)=U^{\xi^{-1}\eta}$ for any $\xi,\eta\in H$.
		\end{itemize}
		If the sets $(V_j)_{j\in J}$ were chosen sufficiently small, then, by \ref{thm2Varouchas}, for every $j\in J, \eta\in H$, there exists an $(m,m)$-form $\chi_j^\eta$ on $U_j^\eta$ such that 
		$$\widetilde{\omega}^{m+1}_{\restriction U_j^\eta}=\ei\partial\overline{\partial}\chi_j^\eta.$$
		Let $\xi\in H$. Since, $\xi^*\widetilde{\omega}=e^{c_\xi}\widetilde{\omega}$, we have
		\begin{multline*} \ei\partial\overline{\partial}(\xi^* \chi_j^\eta) = \xi^*(\ei\partial\overline{\partial}\chi_j^\eta) = \xi^*(\widetilde{\omega}^{m+1}_{\restriction U_j^\eta})= (\xi^*\widetilde{\omega})^{m+1}_{\restriction U_j^{\xi^{-1}\eta}} = \\
			= e^{(m+1)c_\xi}\widetilde{\omega}^{m+1}_{\restriction U_j^{\xi^{-1}\eta}}=  e^{(m+1)c_\xi}\ei\partial\overline{\partial}(\chi_j^{\xi^{-1}\eta})=\ei\partial\overline{\partial}(e^{(m+1)c_\xi}\chi_j^{\xi^{-1}\eta}).
		\end{multline*}
		Hence, we can chose $\chi_j^\eta$ such that they verifiy $\xi^* \chi_j^\eta=e^{(m+1)c_\xi}\chi_j^{\xi^{-1}\eta}$ for any $\xi,\eta\in H$. 
		Then, we define $\widetilde{p}_*\chi_j^\eta:V_j^\eta\rightarrow\mathbb{R}$, 
		$$\widetilde{p}_*\chi_j^\eta(\widetilde{x}^\prime)=\int_{\widetilde{p}^{-1}(\widetilde{x}^\prime)}\chi_j^\eta.$$
		By the above property, we obtain $\xi^* \widetilde{p}_*\chi_j^\eta=e^{(m+1)c_\xi}\widetilde{p}_*\chi_j^{\xi^{-1}\eta}$ for any $\xi,\eta\in H$. Moreover, \ref{prop.3.3.1} together with \ref{prop.3.4.1} ensure that the functions $\widetilde{p}_*\chi_j^\eta$, $j\in J,\eta\in H$, are s.psh (but not necessarily $\mathcal{C}^\infty$) and the difference of any two such functions is pluriharmonic. Finally, applying \cite[Thm.1]{var89}, we obtain a \K\ metric 
		$$\widetilde{\tau}^\prime=\reallywidehat{(V_j^\eta,\psi_j^\eta)_{j\in J, \eta\in H}}$$ 
		on $\widetilde{X^\prime}$. As in \textbf{Case 1}, the choices in the proof of \cite[Thm.1]{var89} can be made such that the property $\xi^*\widetilde{\tau}^\prime=e^{c_\xi}\widetilde{\tau}^\prime$, for every $\xi\in H$ is satisfied, which means that $X^\prime$ admits lcK metrics.
	\end{proof}
	
	\vspace{20pt}
	

\end{document}